\documentclass{amsart}
\usepackage[applemac]{inputenc}
\usepackage{amsmath}
\usepackage{bbm}
\usepackage{latexsym}
\usepackage{xcolor}

\usepackage{tikz}
\usetikzlibrary{arrows}
\usetikzlibrary{patterns}
\usetikzlibrary{shapes}
\usepackage{mathrsfs}
\usepackage{bbm}
\usepackage{enumerate}

\newcommand{\R}{\mathbb{R}}
\newcommand{\RP}{\mathbb{R}\textrm{P}}

\newcommand{\CP}{\mathbb{C}\textrm{P}}
\newcommand{\C}{\mathbb{C}}

\newcommand{\EE}{\mathbb{E}}

\DeclareMathOperator\sgn{sgn}

 \usepackage[usenames,dvipsnames]{pstricks}
 \usepackage{epsfig}
 \usepackage{pst-grad} 
 \usepackage{pst-plot} 

\newtheorem{thm}{Theorem}
\newtheorem{lemma}[thm]{Lemma}
\newtheorem{cor}[thm]{Corollary}
\newtheorem{prop}[thm]{Proposition}
\newtheorem{claim}{Claim}

\theoremstyle{remark} 
\newtheorem{remark}[]{Remark}

\newcommand{\be}{\begin{equation}}
\newcommand{\ee}{\end{equation}}

\usepackage{mathtools} 
\mathtoolsset{showonlyrefs,showmanualtags} 

\ifx\pdfpageheight\undefined
   \usepackage[dvips,colorlinks=true,linkcolor=blue,citecolor=blue,%
      urlcolor=green]{hyperref}
   \usepackage[dvips]{graphicx}
   \makeatletter
   \edef\Gin@extensions{\Gin@extensions,.mps}
   \makeatother
\else
   \usepackage[bookmarksopen=false,pdftex=true,breaklinks=true,%
      backref=page,pagebackref=true,plainpages=false,%
      hyperindex=true,pdfstartview=FitH,colorlinks=true,%
      pdfpagelabels=true,colorlinks=true,linkcolor=blue,%
      citecolor=blue,urlcolor=green,hypertexnames=false%
      ]%
   {hyperref}
\fi
\usepackage{bm}

\title[Random fields and enumerative geometry]{Random fields and the enumerative geometry of lines on real and complex hypersurfaces}
\author{Saugata Basu}
\address{Purdue University}
\email{sbasu@purdue.edu}

\author{Antonio Lerario}
\address{SISSA (Trieste) and Florida Atlantic University}
\email{lerario@sissa.it}

\author{Erik Lundberg}
\address{Florida Atlantic University}
\email{elundber@fau.edu}

\author{Chris Peterson}
\address{Colorado State Univeristy}
\email{peterson@math.colostate.edu }

\begin{document}

\begin{abstract}We derive a formula expressing the average number 
$E_n$ of real lines on a random hypersurface of degree $2n-3$ in $\RP^n$ 
in terms of the expected modulus of the determinant of a special random matrix. 
In the case $n=3$ we prove that the average number of real lines on a random cubic surface in $\RP^3$ equals: 
\be E_3=6\sqrt{2}-3.\ee
Our technique
can also be used to 
express the number $C_n$ of complex lines on 
a \emph{generic} hypersurface of degree $2n-3$ in $\CP^n$ 
in terms of the determinant of a random Hermitian matrix. 
As a special case we obtain a new proof of the classical statement $C_3=27.$ 

We determine, at the logarithmic scale, the asymptotic of the quantity $E_n$, by relating 
it to $C_n$ (whose asymptotic has been recently computed in \cite{seq}). 
Specifically we prove that:
\be \lim_{n\to \infty}\frac{\log E_n}{\log C_n}=\frac{1}{2}.\ee

Finally we show that this approach can be used to compute the number $R_n=(2n-3)!!$ of real lines, counted with their \emph{intrinsic signs} (as defined in \cite{OkTel}), on a generic real hypersurface of degree $2n-3$ in $\RP^n$. \end{abstract}
\maketitle
\section{Introduction}\subsection{Lines on a random cubic}
One of the most classical statements 
from enumerative geometry asserts that there are $27$ lines lying on a generic cubic surface in $\CP^3$: Cayley proved this result in the nineteenth century \cite{Cayley}. Today many proofs exist, some more geometric (using the isomorphism between the blowup of $\CP^2$ at six generic points and the generic cubic in $\CP^3$),  others more topological (exploiting properties of Chern classes), see for example \cite{Dolgachev, 3264}. (A new proof, based on a probabilistic argument, will be given in this paper, see Corollary \ref{cor:C} below.) 

In this paper we will mostly be interested in similar problems over the reals, where the situation is more complicated. Schl\"afli \cite{Schlafli}
showed that if the cubic surface is \emph{real} and smooth then the number of real lines lying on it is 
either 27, 15, 7, or 3. In the blow-up point of view, these correspond respectively to the following four cases for the blow-up points: all of them are real (27); 4 are real and the other two are complex conjugate (15); 2 are real and there are two pairs of complex conjugate (7); there are 3 pairs of complex conjugate (3). Reading the number of real lines from the coefficients of 
the polynomial $f$ defining the cubic is a 
difficult problem.
It is interesting to ask for a probabilistic treatment:
\be\label{q:intro}\emph{``How many real lines are expected to lie on a random real cubic surface in $\mathbb{R}\emph{P}^3$?''}.\ee

To make this question rigorous we should clarify what we mean by ``random''. 
Here we will sample $f$ from the so-called \emph{Kostlan} ensemble. 
We endow the space $\R[x_0, x_1, x_2, x_3]_{(3)}$ 
of real homogeneous polynomials of degree $3$
with a probability distribution by
defining a random polynomial $f$ as a linear combination:
\be f=\sum_{|\alpha|=3}\xi_{\alpha}x_0^{\alpha_0}x_1^{\alpha_1}x_2^{\alpha_2}x_3^{\alpha_3}, \quad \alpha=(\alpha_0, \ldots, \alpha_3),\ee
where the coefficients $\xi_\alpha$ are real, independent, 
centered Gaussian variables with variances 
\be \sigma_\alpha^2=\frac{3!}{\alpha_0!\alpha_1!\alpha_2!\alpha_3!}.\ee
(A similar definition is considered for homogeneous polynomials of degree $d$ in 
several variables, see Section \ref{sec:kostlan} below.) 
This probability distribution is $O(4)$-invariant,
that is, it is invariant under an orthogonal change of variables 
(so there are no preferred points or directions in $\RP^3$).
Moreover it is the only $O(4)$-invariant Gaussian 
probability distribution on $\R[x_0, x_1, x_2, x_3]_{(3)}$ for which $f$ 
can be defined as a
linear combination of monomials with independent Gaussian coefficients. This distribution is also natural from the point of view of algebraic geometry, as it can be equivalently obtained by sampling a random polynomial uniformly from the (projectivization of the) space of real polynomials with the metric induced by the inclusion in the space of complex polynomials with the Fubini-Study metric, see Section \ref{sec:kostlan} below for more details.

Let $Gr(2,4)$ denote the Grassmannian of two dimensional subspaces of $\R^4$ (or equivalently, the Grassmannian of lines in $\RP^3$).
Denoting by $\tau_{2,4}$ the tautological bundle of $Gr(2,4)$ we see that the homogeneous cubic polynomial $f$ defines, by restriction, a section $\sigma_f$ of $\textrm{sym}^3(\tau^{*}_{2, 4})$. A line $\ell$ lies on the cubic surface defined by $\{f=0\}$ if and only if $\sigma_f(\ell)=0$. In this way, when $f$ is random, $\sigma_f$ is a random section of $\textrm{sym}^3(\tau^{*}_{2, 4})$  and the section $\sigma_f$ can be considered as a generalization of a \emph{random field}. In this case, the random field takes values in a vector bundle but on a trivializing chart it reduces to the standard construction (see \cite{AdlerTaylor}). There is a powerful technique that allows one to compute the average number of zeros of a random field, the so called \emph{Kac-Rice formula} \cite{AdlerTaylor}. The adaptation of this technique to random sections of vector bundles is straightforward (we will discuss it in the proofs below) 
and yields the existence of a function $\rho:Gr(2,4)\to \R$ (the \emph{Kac-Rice density}) such that the answer to question \eqref{q:intro} can be written as:
\be E_3=\int_{Gr(2,4)}\rho\,\omega_{Gr(2,4)}\ee
(here $\omega_{Gr(2,4)}$ is the volume density of $Gr(2,4)$ with respect to the metric induced by the Pl\"ucker embedding). In our case, by invariance of the problem under the action of the 
orthogonal group, $\rho$ is a constant function.
However, the computation of this constant is still delicate 
and is one of the main results of this paper.  To be precise, we show in Theorem \ref{thm:cubic} that its value equals:
\be \rho= \frac{6\sqrt{2}-3}{2\pi^2}.\ee
Since the volume of $Gr(2,4)$ is equal to $2\pi^2$, we are able to conclude that $E_3=6\sqrt{2}-3.$

Notice that over the complex numbers
whatever reasonable (i.e. absolutely continuous 
with respect to Lebesgue measure) probability distribution 
we consider on the space of the coefficients of the defining polynomial, 
a random complex cubic surface will be generic in the sense of algebraic geometry with probability one (this is essentially due to the fact that complex discriminants have real codimension two). 
For this reason our technique, 
which produces an \emph{expected} answer over the reals, 
gives the \emph{generic} answer when adapted to the complex setting.

If we take the coefficients of $f$ to be \emph{complex} Gaussians 
(with the same variances as for the real case) we get a \emph{complex Kostlan} polynomial. Equivalently we can sample from the projectivization of the space of complex polynomials with the natural Fubini-Study metric (see Section \ref{sec:kostlan} below). In this way $\sigma_f$ is a random holomorphic section of $\textrm{sym}^{3}(\tau^*_{2,4})$ where $\tau_{2,4}$ is now the tautological bundle of $Gr_{\C}(2,4)$ (the Grassmannian of complex two dimensional subspaces of $\C^4$). Here again one can use the Kac-Rice approach, interpreting $4$ complex variables as $8$ real variables and $C_3$ can be rewritten as an integral (with respect to the volume density induced by the complex Pl\"ucker embedding) over $Gr_\C(2,4)$ of a function $\rho_\C$. 
The model is invariant under the  action of $U(4)$ so that $\rho_\C$ is constant. 
The evaluation of this constant is now much easier than its real analogue 
(see Corollary \ref{cor:C}), and it seems more combinatorial in nature. We compute the value to be:
\be \rho_\C\equiv \frac{324}{\pi^4}.\ee
The volume of $Gr_\C(2,4)$ is equal to $\frac{\pi^4}{12}$ and this implies $C_3=27$.

Going back to the real case, Segre \cite{Segre} divided the real lines lying on a cubic surface into {\it hyperbolic} lines and {\it elliptic} lines, and showed that the difference between the number, $h$, of hyperbolic lines and 
the number, $e$, of elliptic lines always satisfy the relation $h-e=3$. In this direction we note that the technique below can also be used to compute this difference; this will correspond to a signed Kac-Rice density and its computation amounts to removing the modulus in the expectation of the determinant of the matrix in Theorem \ref{thm:averagereal}. Here we know that the signed count is an invariant number $R_3$ of the generic cubic \cite{OkTel} and consequently that the average answer is the generic one; the explicit computation $R_3=\frac{3}{2}\EE\det \hat{J}_3=3$ recovers this number.

\begin{remark}Allcock, Carlson, and Toledo \cite{ACD} have studied the moduli space of real cubic surfaces form the point of view of hyperbolic geometry. They compute the orbifold Euler characteristic (which is proportional to the hyperbolic volume) of each component of the moduli space. One could take the weighted average of the number of real lines, weighted by the volume of the corresponding component, as an ``average count''. This yields the number $\frac{239}{37}$, see \cite[Table 1.2]{ACD}. However this way of counting does not have a natural probabilistic interpretation and generalization to $n> 3$.
\end{remark}
\subsection{Lines on hypersurfaces}\label{sec:discussion}
The same scheme can be applied to the problem of enumeration of lines on hypersurfaces of degree $2n-3$ in $\RP^n$ and $\CP^n$ (if the degree is larger than $2n-3$ then a random hypersurface is too ``curved'' and will contain no lines
while if the degree is smaller than $2n-3$ then lines will appear in families). 
Theorem \ref{thm:averagereal} and its complex analogue Theorem \ref{thm:averagecomplex} give a general recipe for the computation of the average number $E_n$ of real lines on a real Kostlan hypersurface and the number $C_n$ of lines on a generic complex hypersurface.
Zagier \cite{seq} showed that $C_n$ can be computed as the coefficient of $x^{n-1}$ in the polynomial 
\be p_n(x)=(1-x) \prod_{j=0}^{2n-3} (2n-3-j+jx) \ee and found that the asymptotic behavior for $C_n$ as $n\to \infty$ is given by:
\be \label{thm:DZ} C_n\sim \sqrt{\frac{27}{\pi}}(2n-3)^{2n-\frac{7}{2}}(1+O(n^{-1})).\ee
One could notice that the expression that we derive for $C_n$ in Theorem \ref{thm:averagecomplex} looks almost like the square of the analogous expression for $E_n$ from Theorem \ref{thm:averagereal}. 
In Theorem \ref{thm:sqrt} we will make this comparison more rigorous, by proving that the two quantities are related by:
\be\label{eq:sqrt} \lim_{n\to \infty}\frac{\log E_n}{\log C_n}=\frac{1}{2}.\ee
This is roughly saying that the average number of real lines on a random real hypersurface is the ``square root'' of the number of complex lines (this is true up to an exponential factor; note that the number $E_n$ and $C_n$ are super-exponential). 

After this paper was written, it was brought to our attention 
in a private communication
that Peter B\"urgisser conjectured the 
square root law (as stated in Theorem \ref{thm:sqrt})
in a talk given at a conference in 2008, at Bernoulli Center in Lausanne, 
where he also gave numerical evidence with regards to the value of $E_3$.

Remarkably, Okonek and Teleman \cite{OkTel} and Finashin and Kharlamov \cite{finkh} have shown that over the reals it is still possible to associate a numerical invariant $R_n$ to a generic hypersurface, counting the number of real lines with canonically assigned \emph{signs}. 
Even though the bundles $\textrm{sym}^{2n-3}(\tau^*_{2, n+1})$ and $TGr(2, n+1)$ might not be orientable themselves, Okonek and Teleman \cite{OkTel} introduce a notion of \emph{relative orientation} between them that allows one to define an Euler number up to a sign. In this case it turns out there is a canonical relative orientation, and each zero of a generic section of $\textrm{sym}^{2n-3}(\tau^*_{2, n+1})$ comes with an \emph{intrinsic sign} (see Section \ref{sec:lower} below). The relative Euler number has the usual property that the sum of the zeroes of a generic section counted with signs is invariant; for the problem of lines on hypersurfaces it equals: 
\be R_n=(2n-3)!!.\ee
In particular note that $E_n\geq (2n-3)!!$. Using the random approach, in Proposition \ref{propo:signed} we give an alternative proof of $R_n=(2n-3)!!$.

We summarize this discussion by stating the following inequalities, that offer a broad point of view (here $b>1$ is a universal constant whose existence is proved in Proposition \ref{propo:upper}):
\be R_n\leq E_n\leq b^n C_n^{1/2}.\ee

Future work will include other applications of the probabilistic approach to complex enumerative and random real enumerative geometry.

\subsection*{Acknowledgements}
The question \eqref{q:intro} motivating this paper was posed to the second author by F. Sottile. 
This paper originated during the stay of the authors at SISSA (Trieste), 
supported by Foundation Compositio Mathematica.
\section{The real case}
\subsection{Real Kostlan polynomials}\label{sec:kostlan}

 
Let $E=\R^{n+1}$. Then, $\textrm{sym}^d(E^*)$ is identified with the space,
$\R[x_0, x_1,\ldots, x_n]_{(d)}$, 
of homogeneous polynomials of degree $d$ in $n$ variables.
A Gaussian ensemble can be specified by choosing 
a scalar product on 
$\textrm{sym}^d(E^*)$
in which case
$f$ is sampled according to the law:
$$\text{Probability} (f\in A) = \frac{1}{v_{n,d}}\int_A e^{-\frac{\|f \|^2}{2}} df, $$
where $v_{n,d}$  is a normalizing constant that makes the integrand into a probability density function,
and $df$ is the volume form on 
$\textrm{sym}^d(E^*)$
induced by the chosen scalar product. 
The ensemble we will consider, known as the \emph{Kostlan ensemble},
results from choosing as a scalar product 
the Bombieri product\footnote{This inner product
has also been referred to as 
the ``Fischer product'', especially in the field of holomorphic PDE
after H.S. Shapiro made a detailed study \cite{Sh} reviving
methods from E. Fischer's 1917 paper \cite{Fi}. 
The names of 
H. Weyl, V. Bargmann, and V. A. Fock have also been attached to this inner product.}, defined as:
 $$\langle f, g \rangle_B = \frac{1}{d! \pi^{n+1}}\int_{\C^{n+1}} f(z) \overline{g(z)} e^{-\|z\|^2} dz.$$
The following expression relates the Bombieri norm of $f=\sum_{|\alpha|=d}f_\alpha z_1^{\alpha_0}\cdots z_n^{\alpha_n}$ with its coefficients in the monomial basis (see \cite[Equation (10)]{NewmanShapiro}):
\be
\|f\|_B=\left(\sum_{|\alpha|=D}|f_\alpha|^2\frac{\alpha_0!\cdots \alpha_n!}{d!}\right)^{\frac{1}{2}}.\ee
The Bombieri product  defined above can be described in a more algebro-geometric language as follows. The standard Hermitian inner product on 
$\mathbb{C}^{n+1}$ induces a Hermitian inner product, $(\cdot,\cdot)_d$ on the line bundles $\mathcal{O}_{\CP^n}(d)$, and this in turn induces a
Hermitian inner product on the finite-dimensional space of holomorphic sections, $H^0(\mathcal{O}_{\CP^n}(d))$,  by
\[
\langle f,g \rangle = \int_{\CP^n} (f(z),g(z))_d \ dz,
\]
where the integration is with respect to the volume form on $\CP^n$ induced by the standard Fubini-Study metric on $\CP^n$.
The restriction of this last Hermitian inner product to real sections of $\mathcal{O}_{\CP^n}(d)$ agrees up to normalization by a constant with the Bombieri product 
(after identifying $\R[x_0,x_1,\ldots, x_n]_{(d)}$ with the space of real holomorphic sections of  $\mathcal{O}_{\CP^n}(d)$). Thus, amongst all possible  $O(n+1)$-invariant
measures on $\R[x_0,x_1,\ldots, x_n]_{(d)}$, the Kostlan measure seems to be the most natural one from the point of view of algebraic geometry.

An equivalent, and often more practical, approach to Gaussian ensembles is to build 
$f$ as a linear combination, using independent Gaussian coefficients,  of the vectors in
an orthonormal basis for the associated scalar product.
The Kostlan ensemble has the distinguished property of being 
the unique Gaussian ensemble that is invariant under any orthogonal 
changes of variables while simultaneously having the monomials as an orthogonal basis.
We can sample $f$ by placing independent Gaussians $\xi_\alpha$ 
in front of the monomial basis (here $\alpha$ is a multi-index):
$$f(x)=\sum_{|\alpha|=d}\xi_\alpha x^{\alpha}, \quad \xi_\alpha\sim N\left(0,\binom{d}{\alpha} \right).$$ 
Note that there are other orthogonally-invariant models, but their
description requires special functions (spherical harmonics), see \cite{Kostlan95, FLL}.

\subsection{A general construction}
Consider a random vector $v=(v_1, \ldots, v_{2n-3})$ in $\R^{2n-3}$ whose entries are independent Gaussian variables distributed as: 
\be 
v_j\sim N\left(0,\binom{2n-4}{j-1}\right)\quad j=1, \ldots,  2n-3.\ee
Let now $v^{(1)}, \ldots, v^{(n-1)}$ be independent random vectors all distributed as $v$. We define the random $(2n-2)\times (2n-2)$ matrix $\hat{J}_n$ as:
\be\hat J_n=\left[\begin{array}{ccccc}v_1^{(1)} & 0 & \ldots & v_{1}^{(n-1)} & 0 \\v_2^{(1)} & v_1^{(1)} &  & v_2^{(n-1)} & v_{1}^{(n-1)} \\\vdots & v_2^{(1)} &  & \vdots & v_2^{(n-1)} \\v_{2n-3}^{(1)} & \vdots &  & v_{2n-3}^{(n-1)} & \vdots \\0 & v_{2n-3}^{(1)} & \ldots & 0 & v_{2n-3}^{(n-1)}\end{array}\right]. \ee
\begin{thm}\label{thm:averagereal}The average number $E_n$ of real lines on a random Kostlan hypersurface of degree $2n-3$ in $\mathbb{R}\emph{P}^n$ is
\be E_n=\left(\frac{(2n-3)^{n-1}}{\Gamma(n)}\prod_{k=0}^{2n-3}\binom{2n-3}{k}^{-1/2}\right)\EE|\det \hat J_n|.\ee
\end{thm}
\begin{proof}We will start with a general construction that works for the Grassmannian $Gr(k,m)$ and only at the end of the proof we will specialize to the case $k=2, m=n+1.$

For $i=1, \ldots, k$ and $j=k+1, \ldots, m$ consider the matrix $E_{ij}\in \mathfrak{so} (m)$ which consists of all zeros except for having a $1$ in position $(i,j)$ and a $-1$ in position $(j,i)$.
Then $e^{tE_{ij}}\in O(m)$ is the clockwise rotation matrix of an angle $t$ in the plane, $\textrm{span}\{e_i, e_j\}$, spanned by the $i^{th}$ and $j^{th}$ standard basis vectors. Let  now $t=(t_{ij})\in \R^{k(m-k)}$ and consider the map $R:\mathbb{R}^{k(m-k)}\times \textrm{span}\{e_1, \ldots, e_k\}\to \R^m$ defined by:
\be \label{eq:RRR} R(t, y)=\left(e^{\sum_{ij}{t_{ij}E_{ij}}}\right)\cdot y=R_t\cdot y.\ee
Let $Gr^+(k,m)$ denote the Grassmannian of \emph{oriented} $k$-planes in $\R^m$. As a Riemannian manifold $Gr^{+}(k,m)$ is identified with its image under the spherical Pl\"ucker embedding. The image is equal to the set of simple, norm-one vectors in $\Lambda^{k}(\R^m)$; moreover the Riemannian metric induced on $G^{+}(k,m)$ from this embedding is the same as the metric induced by declaring the quotient map $SO(m)\to G^{+}(k,m)$ to be a Riemannian submersion (see \cite{Kozlov1} for more details).

The map $\psi:\R^{k(m-k)}\to Gr^+(k,m)$ defined by:
\be \psi(t)=R_te_1\wedge \cdots \wedge R_te_k\ee
is a local parametrization of $Gr^{+}(k,m)$ near $e_1\wedge\cdots \wedge e_k=\psi(0)$. In fact the map $\psi$ is the Riemannian exponential map centered at $e_1\wedge \cdots\wedge e_k$ (see \cite{Kozlov1}). As a consequence, the map $\varphi:U\to \R^{k(m-k)}$ (the inverse of $\psi$) gives a coordinate chart on a neighborhood $U$ of $e_1\wedge \cdots\wedge e_k$.

Let now $\tau^*_{k,m}$ denote the dual of the tautological bundle (on $Gr^{+}(k,m)$) and consider its symmetric $d$-th power $\textrm{sym}^d(\tau^*_{k,m})$. Note that a homogeneous polynomial $f\in \R[x_1, \ldots, x_m]_{(d)}$ defines a real holomorphic section $\sigma_f$ of the bundle $\textrm{sym}^d(\tau^*_{k,m})$, simply by considering the restriction $\sigma_f(w)=f|_w$ (here $w$ denotes the variable in $Gr^{+}(k,m)$). 
By construction, $R(\varphi(w), \cdot)^*f$  defines for every $w\in Gr^{+}(k,m)$ a polynomial in $\R[y_1, \ldots, y_k]_{(d)}$. As a consequence, the map
\be q:\textrm{sym}^d(\tau^*_{k,m})|_U\to U\times \R[y_1, \ldots, y_k]_{(d)}\ee
defined by:
\be f\mapsto \left(w, R(\varphi(w), \cdot)^*f\right),\quad f\in \textrm{sym}^d(\tau^*_{k,m})|_w \ee
gives a trivialization for $\textrm{sym}^d(\tau^*_{k,m})$ over $U$.

Since $Gr^{+}(k,m)$ is compact and connected, the map $\psi$ (the Riemannian exponential map) is surjective and we can choose $U$ such that $Gr^{+}(k,m)\backslash U$ has measure zero (in this way, integrating a continuous function over $U$ with respect the volume density of $Gr^{+}(k,m)$ gives the same result as integrating the continuous function over all of $Gr^{+}(k,m)$).

Let $f\in \R[x_1,\ldots, x_m]_{(d)}$ be a random Kostlan polynomial. Using $f$ we can define the random map:
\be\label{eq:coordinates} \tilde\sigma_f:\R^{k(m-k)}\to \R[y_1, \ldots, y_k]_{(d)}\simeq \R^{\binom{k+d-1}{d}}\ee
where $\tilde\sigma_f$ is nothing but $\sigma_f$ in the trivialization given by $q$.

Assume now that $k(m-k)=\binom{k+d-1}{d},$ so that $\textrm{rank}\left(\textrm{sym}^d(\tau^*_{k,m})\right)=\dim Gr(k,m)$. Then we can use the Kac-Rice formula \cite[Theorem 12.1.1]{AdlerTaylor} to count the average number of zeros of $\tilde{\sigma}_f$:
\be \EE\#\{\tilde{\sigma}_f=0\}=\int_{U}\EE\left\{|\det J(w)|\,\big|\, \sigma_f(w)=0\right\}p(0;w)\psi^*\omega_{Gr^{+}(k,m)},\ee
where $p(0;w)$ is the joint density at zero of the random vector $\sigma_f(w),$ $\omega_{Gr^{+}(k,m)}$ is the volume form induced by the spherical Pl\"ucker embedding and:
\be J(w)=\left(\nabla_i(\tilde{\sigma}_f)_j(w)\right)\ee
is the matrix of the derivatives of the components of $\tilde{\sigma}_f$ with respect to an orthonormal frame field at $w$. Since the polynomial $f$ is $O(m)$-invariant,  the quantity:
\be \rho(w)=\EE\left\{|\det J(w)|\,\big|\, \sigma_f(w)=0\right\}p(0;w)\ee
does not depend on $w\in U$ thus $\rho=\rho(w)$ is constant.
As a consequence:
\be \EE\#\{\tilde{\sigma}_f=0\}=|Gr^+(k,m))|\rho.\ee
Since $Gr^{+}(k,m)$ is a Riemannian double covering of $Gr(k,m)$ and the number of zeros of $\sigma_f$ as a section of the bundle $\textrm{sym}^d(\tau^*_{k,m})$ on $Gr^+(k,m)$ is twice the number of zeros as a section of $\textrm{sym}^d(\tau^*_{k,m})$ on  $Gr(k,m)$, then:
\be \EE\#\{w\in Gr(k,m)\,|\, \sigma_f(w)=0\}=|Gr(k,m)|\rho.\ee
Let  $w_0=e_1\wedge \cdots \wedge e_k$. It is easy to compute the density at zero of $\tilde{\sigma}_f(w_0)$. In fact, in the trivialization $q$ we have:
\be \tilde{\sigma}_f|_{w_0}=\sum_{|\alpha|=d}\xi_{\alpha_1, \ldots, \alpha_k, 0, \ldots, 0} y_1^{\alpha_1}\cdots y_k^{\alpha_k}.\ee
Hence the coefficients of $ \tilde{\sigma}_f|_{w_0}$ are independent random variables and are still Kostlan distributed
(in fact, the restriction of a Kostlan polynomial to a subspace is still a Kostlan polynomial). 
As a consequence:
\be p(0;w_0)=\prod_{\alpha_1+\cdots+\alpha_k=d}\left(\frac{\alpha_1!\cdots \alpha_k!}{d!2\pi }\right)^{1/2}.\ee
For the computation of $\EE\left\{|\det J(w)|\,\big|\, \sigma_f(w)=0\right\}$ we proceed as follows.
For $i=1, \ldots, k$ and $j=k+1, \ldots, m$, consider the curve $\gamma_{ij}:(-\epsilon, \epsilon)\to Gr^{+}(k,m)$:
\be \gamma_{ij}(s)=e^{sE_{ij}}e_1\wedge \cdots\wedge e^{sE_{ij}}e_k.\ee
For the curve $\gamma_{ij}$, note that:
\be \gamma_{ij}(0)=w_0\quad \textrm{and}\quad  \dot{\gamma}_{ij}(0)=\psi_*(\partial_{t_{ij}}).\ee
By construction,
$$\{\psi_*(\partial_{t_{ij}})\, |\,i=1, \ldots, k, j=k+1, \ldots, m\}$$
is an orthonormal basis for $T_{w_0}Gr^{+}(k,m)$, hence we can differentiate along the curves $\gamma_{ij}$ for computing the matrix $J$. By definition of the trivialization $q$ we have:
\be \tilde{\sigma}_f(\gamma_{ij}(s))(y_1, \ldots, y_k)=f\left(e^{sE_{ij}}y\right), \quad y=(y_1, \ldots, y_k, 0, \ldots, 0).\ee 
Note that:
\be e^{sE_{ij}}y=(y_1, y_2, \ldots, y_{i-1}, \underbrace{y_i \cos s}_{\textrm{$i$-th entry}}, y_{i+1}, \ldots, y_k, 0, \ldots, 0, \underbrace{y_i\sin s}_{\textrm{$j$-th entry}}, 0, \ldots, 0).\ee
In particular:
\begin{align}\tilde{\sigma}_f(\gamma_{ij}(s))(y)&=\sum_{|\alpha|=d} \xi_\alpha y_1^{\alpha_1}\cdots y_{i-1}^{\alpha_{i-1}}(y_i\cos s)^{\alpha_i}(y_i\sin s)^{\alpha_j}y_{i+1}^{\alpha_{i+1}}\cdots y_k^{\alpha_k}\\
&=\sum_{|\alpha|=d}\xi_\alpha(\cos s)^{\alpha_i} (\sin s)^{\alpha_j}y_1^{\alpha_1}\cdots y_{i-1}^{\alpha_{i-1}}y_{\alpha_{i}}^{\alpha_i+\alpha_j}y_{i+1}^{\alpha_{i+1}}\cdots y_k^{\alpha_k}.
\end{align}
Taking the derivative of each coefficient with respect to $s$ and evaluating at $s=0$ we get:
\begin{align}\frac{d}{ds}\left(\tilde{\sigma}_{f}(\gamma_{ij}(s))(y)\right)\big|_{s=0}&=\sum_{\alpha_1+\cdots+\alpha_k+1=d}\xi_\alpha y_1^{\alpha_1}\cdots y_{i-1}^{\alpha_{i-1}}y_i^{\alpha_i+1}y_{i+1}^{\alpha_{i+1}}\cdots y_k^{\alpha_k}.\\
&=\sum_{|\beta|=d}q_\beta y_1^{\beta_1}\cdots y_k^{\beta_k}\end{align}
from which we deduce that the coefficient $q_\beta$ is distributed as:
\be q_\beta=\xi_{\beta_1, \ldots, \beta_{i-1}, \beta_i+1, \beta_{i+1}, \ldots, \beta_k, 0, \ldots, 0, 1, \ldots, 0}\quad \textrm{(there is a $1$ in position $j$)}.\ee
From this we immediately see that $|\det J(w_0)|$ and $\tilde{\sigma}_f(w_0)$ are independent random variables (because of the ``$1$ in position $j$'' above) and:
\be \EE\left\{|\det J(w_0)|\,\big|\, \sigma_f(w_0)=0\right\}= \EE\left\{|\det J(w_0)|\right\}.\ee
The matrix $J(w_0)$ is in general quite complicated, but when we specialize to the case $k=2$ it becomes simpler (this is due to the fact that there is a natural way to order the monomials of a one-variable polynomial).  In fact, working in the above basis $\{\partial_{1,3}, \partial_{2,3}, \ldots, \partial_{1,m}, \partial_{2, m}\}$ for $T_{w_0}G^+(2, m)$ and $\{y_1^d, y_1^{d-1}y_2, \ldots, y_1y_2^{d-1}, y_2^d\}$ for $\R[y_1, y_2]_{(d)}$, we see that:
\be \partial_{1,j}\tilde{\sigma}_f=\sum_{|\alpha|=d, \alpha_j=1}\xi_\alpha y_1^{1+\alpha_1}y_2^{\alpha_2}\quad \textrm{and}\quad \partial_{2,j}\tilde{\sigma}_f=\sum_{|\alpha|=d, \alpha_j=1}\xi_\alpha y_1^{\alpha_1}y_2^{1+\alpha_2}.
\ee
For example, in the case $m=4$ the matrix $J(w_0)$ is:
$$ J (w_0)= \begin{bmatrix}
\xi_{2010} & 0 & \xi_{2001} & 0 \\
\xi_{1110} & \xi_{2010} & \xi_{1101} & \xi_{2001} \\
\xi_{0210} & \xi_{1110} & \xi_{0210} & \xi_{1101} \\
0 & \xi_{0210} & 0 & \xi_{0201} \\
\end{bmatrix} ,$$

In the case $k=2, m=n+1$, for the Grassmannian  $Gr(2,n+1)$ of lines in $\RP^n$, the matrix $J(w_0)$ has the same shape as $\hat J_n$; moreover collecting $\sqrt{2n-3}$ from each row of $J(w_0)$ we get a matrix that is distributed as $\hat{J}_n$:
\be \EE\left\{|\det J(w_0)|\right\}=(2n-3)^{\frac{2n-2}{2}}\EE|\det \hat{J}_n|.\ee
Thus the average number, $E_n$, of real lines on a random Kostlan hypersurface of degree $2n-3$ in $\mathbb{R}\textrm{P}^n$ is:
\begin{align}E_n&=|Gr(2, n+1)|\ p(0;w_0)\ (2n-3)^{\frac{2n-2}{2}}\ \EE|\det \hat{J}_n|\\
&=\frac{\pi^{n-\frac{1}{2}}}{\Gamma\left(\frac{n}{2}\right)\Gamma\left(\frac{n+1}{2}\right)}\prod_{k=0}^{2n-3}\left(\frac{k!\cdots (2n-3-k)!}{(2n-3)!2\pi }\right)^{1/2}(2n-3)^{\frac{2n-2}{2}}\ \EE|\det \hat{J}_n|\\
&=\frac{(2n-3)^{n-1}}{\Gamma(n)}\prod_{k=0}^{2n-3}\binom{2n-3}{k}^{-1/2}\EE|\det \hat J_n|,
\end{align}
where for the volume of the Grassmannian we have used the formula (see Remark \ref{remark:vg} below):
\be |Gr(2, n+1)|=\frac{\pi^{n-\frac{1}{2}}}{\Gamma\left(\frac{n}{2}\right)\Gamma\left(\frac{n+1}{2}\right)}.\ee\end{proof}

\begin{remark}[Volume of Grassmannians]\label{remark:vg}The formula $|Gr(2, n+1)|=\frac{\pi^{n-\frac{1}{2}}}{\Gamma\left(\frac{n}{2}\right)\Gamma\left(\frac{n+1}{2}\right)}$ follows from $|Gr(k,m)|=\frac{|O(m)|}{|O(k)||O(m-k)|}$ and the formula for the volume of the Orthogonal group (see \cite[Section 3.12]{Howard}):
\be  |O(k)|=\frac{2^k \pi^{\frac{k^2+k}{4}}}{\Gamma(k/2)\Gamma((k-1)/2)\cdots\Gamma(1/2)}
  .\ee
 Note that an analogous formula holds for the complex Grassmannian:
  \be |Gr_{\C}(k,m)|=\frac{|U(m)|}{|U(k)||U(m-k)|}\quad \textrm{and}\quad |U(k)| = \frac{2^k \pi^{\frac{k^2+k}{2}}}{\prod_{i=1}^{k-1} i!}.\ee
  
\end{remark}
\subsection{Intrinsic signs and lower bound}\label{sec:lower}
Generalizing Segre's observation $h-e=3$, Okonek and Teleman \cite{OkTel} have shown that to each real line $\ell$ on a generic real hypersurface $H$ of degree $2n-3$ in $\RP^n$ it is possible to canonically associate a sign $\epsilon(\ell)$ which satisfies the property  that $\sum_{\ell\subset H}\epsilon(\ell)$ does not depend on $H$, and that this number equals $(2n-3)!!$.
 The crucial observation is that  there exists a line bundle $L\to Gr(2,n+1)$ such that \cite[Proposition 12]{OkTel}:
\be K= \det(\textrm{sym}^{2n-3}(\tau^*_{2, n+1}))\otimes \det(TGr(2,n+1))= L\otimes L.\ee
Hence $K$ has a trivialization and  $\textrm{sym}^{2n-3}(\tau^*_{2, n+1})$ and $TGr(2,n+1)$ are said to be \emph{relatively oriented} (even if $\textrm{sym}^{2n-3}(\tau^*_{2, n+1})$  and the Grassmannian themselves might not be orientable). Relative orientation is all one needs to introduce an Euler number, which has the usual properties and is well defined up to a sign \cite[Lemma 5]{OkTel}. The fact that $K=L\otimes L$ allows in this specific case to canonically choose the relative orientation and the sign $\epsilon(\ell)$ for $\ell\subset H$ is defined in an intrinsic way. Using this notation Okonek and Teleman \cite{OkTel} and Kharlamov and Finashin \cite{finkh} have independently proved that:
\be R_n\doteq \left|\sum_{\ell\subset H}\epsilon(\ell)\right|=(2n-3)!!.\ee

We now show that using the Kac-Rice formula, and the knowledge that the modulus of a signed count is invariant, one can recover the value of $R_n$.

\begin{prop}\label{propo:signed}Using the above notation, we have:
\be R_n=(2n-3)!!.\ee
\end{prop}
\begin{proof}
Let us fix a trivialization of $\textrm{sym}^{2n-3}(\tau^*_{2, n+1})$ on an open dense set $U\subset Gr(2, n+1)$ as in the proof of Theorem \ref{thm:averagereal}.
The lines on $H=\{f=0\}$ contained in $U$ are the zeroes of:
\be \tilde\sigma_f:\R^{2n-2}\to \R^{2n-2}.\ee
By \cite[Lemma 5]{OkTel} we know that:
\be \label{eq:upto}\sum_{\tilde\sigma_f(w)=0}\textrm{sign}\det (J\tilde\sigma_f)(w)=\epsilon\cdot\sum_{\ell\subset H, \ell\in U}\epsilon(\ell)\ee
and that the choice of the sign $\epsilon=\pm1$ is determined once the trivialization is fixed.
On the other hand, for a random map $F:\R^{2n-2}\to \R^{2n-2}$ (sufficiently smooth, e.g. satisfying the hypothesis of \cite[Theorem 12.1.1]{AdlerTaylor}), we have:
\be \EE \sum_{F(w)=0}\textrm{sign}\det (J F)(w)=\int_{\R^{2n-2}}\EE\{\det (J F)(w)\,|\, F(w)=0\}\cdot p(0; w)dw.\ee
This follows from \cite[Lemma 3.5]{Nicolaescu}, as the function $\textrm{sign}(\det(\cdot)):\R^{(2n-2)\times (2n-2)}\to \R$ is admissible (in the terminology of \cite{Nicolaescu}).
Applying this to the case $F=\tilde\sigma_f$ and arguing exactly as in the proof of Theorem \ref{thm:averagereal} we get:
\be\label{eq:upto2} \EE \sum_{\tilde\sigma_f(w)=0}\textrm{sign}\det (J\tilde\sigma_f)(w)=\left(\frac{(2n-3)^{n-1}}{\Gamma(n)}\prod_{k=0}^{2n-3}\binom{2n-3}{k}^{-1/2}\right)\EE \det \hat J_n\ee
Let us first evaluate the factor $\EE \det \hat J_n$. To simplify notations, let us denote by $x_{i,j}$ the random variable $v_{j}^{(i)}$ in the matrix $\hat J_n$ (we will also use this notation later in Remark \ref{remark1}). After taking expectation, by independence, the only monomials in the expansion of $\det \hat{J}_n$ that give a nonzero contribution are those with all squared variables (for example $ x_{1,1}^2x_{2,3}^2\cdots x_{n-1, 2n-3}^2$). 
These monomials all have the form:
\be x_{i_1, 1}^2x_{i_2, 3}^2\cdots x_{i_{n-1}, 2n-3}^2, \quad \{i_1, \ldots, i_{n-1}\}=\{1, \ldots, n-1\}.\ee
There are $(n-1)!$ many such monomials and because every second subscript is shifted by two, in the expansion of $\det\hat J_n$ they all appear with the same sign. Moreover the product of all the variances of the variables in each of these monomials equal:
\be \prod_{\textrm{$j$ odd}}\binom{2n-4}{j-1}=\prod_{k=1}^{n-1}\binom{2n-4}{2k-2}.\ee
From this we obtain:
\be\label{eq:signed1} \EE \det \hat J_n=(n-1)!\prod_{k=1}^{n-1}\binom{2n-4}{2k-2}.\ee
We look now at the term:
\begin{align} \prod_{k=0}^{2n-3}\binom{2n-3}{k}^{-1/2}&=\prod_{k=0}^{2n-3}\left(\frac{k! (2n-3-k)!}{(2n-3)!}\right)^{1/2}\\
&=\frac{1}{(2n-3)!^{n-1}}\left(\prod_{k=0}^{2n-3} k!\right)^{1/2}\left(\prod_{k=0}^{2n-3}(2n-3-k)!\right)^{1/2}\\
&=\frac{1}{(2n-3)!^{n-1}}\prod_{k=0}^{2n-3} k!.
\end{align}
Collecting all this together, we obtain:
\begin{align}\left(\frac{(2n-3)^{n-1}}{\Gamma(n)}\prod_{k=0}^{2n-3}\binom{2n-3}{k}^{-1/2}\right)\EE \det \hat J_n&=\frac{\left(\prod_{k=0}^{2n-3}k!\right) \left(\prod_{k=1}^{n-1}\binom{2n-4}{2k-2}\right)}{(2n-4)!^{n-1}}\\
&=\left(\prod_{k=0}^{2n-3}k!\right)\left(\prod_{k=1}^{n-1}\frac{1}{(2k-2)!(2n-2k-2)!}\right)\\
&=\left(\prod_{j=1}^{n-1}(2j-1)!\right)\left(\prod_{j=1}^{n-1}\frac{1}{(2n-2k-2)!}\right)\\
&=\left(\prod_{j=1}^{n-1}(2j-1)!\right)\left(\prod_{j=1}^{n-1}\frac{1}{(2j-2)!}\right)\\
&=\prod_{j=1}^{n-1}\frac{(2j-1)!}{(2j-2)!}\\
&=\prod_{j=1}^{n-1}(2j-1)=(2n-3)!!
\end{align}
Combining the last equation with \eqref{eq:upto2} and \eqref{eq:upto} concludes the proof.
\end{proof}
Let us now  denote by $\rho_n$ the number:
\be\rho_n=\frac{(2n-3)^{n-1}}{\Gamma(n)}\prod_{k=0}^{2n-3}\binom{2n-3}{k}^{-1/2}.\ee
Then Proposition \ref{propo:signed} shows that $\rho_n\cdot |\EE\det J_n|= (2n-3)!!$.
As a corollary, since:
\be E_n=\rho_n\cdot \EE |\det J_n|\geq\rho_n\cdot |\EE\det J_n|,\ee 
we derive the following lower bound for $E_n$.
Note however that the proof of $\rho_n\cdot |\EE\det J_n|\geq (2n-3)!!$ does not require \cite[Corollary 17]{OkTel}, hence this lower bound does not depend on the knowledge that the signed count is invariant.
\begin{cor}\label{cor:signed}The following inequality holds:
\be E_n\geq (2n-3)!!\ee
\end{cor}

\section{The average number of real lines on a random cubic}

\begin{thm}\label{thm:cubic}The average number of real lines on a random cubic surface in $\mathbb{R}\emph{P}^3$ is $6 \sqrt{2} - 3$.
\end{thm}
\begin{proof}Applying Theorem \ref{thm:averagereal} for the special case $n=3$ we get:
\be\label{eq:E3} E_3=\frac{3}{2}\EE|\det \hat{J}_3|.\ee
where:
$$ \hat{J_3} = \begin{bmatrix}
a & 0 & d & 0 \\
\sqrt{2}b & a & \sqrt{2}e & d \\
c & \sqrt{2}b & f & \sqrt{2}e \\
0 & c  & 0 & f \\
\end{bmatrix} ,$$
with $a,b,c,d,e,f$ independent standard normal random variables.

In order to compute $\EE |\det \hat{J}_3 |$,
we first observe that:
$$ \det \hat{J}_3 =  (af-cd)^2-2(bf-ce)(ae - bd),$$
which will lead us to work in terms of the random variables:
\begin{align*}
x &= (bf-ce), \\
y &= (af-cd) , \\
z &= (ae - bd),
\end{align*}
in order to compute the expectation of $|\det \hat{J}_3| = |2xz - y^2|$.
We use the method of characteristic functions (Fourier analysis)
in order to compute the joint density $\rho(x,y,z)$ of $x,y,z$.

By the Fourier inversion formula, we have:
\be\label{eq:inversion}
\rho(x,y,z) = \frac{1}{(2\pi)^3} \int_\R \int_\R \int_\R e^{-i(t_1 x  + t_2 y + t_3 z)}\ \hat{\rho}(t_1,t_2,t_3)\ dt_1 dt_2 dt_3,
\ee
where
\be\label{eq:char}
 \hat{\rho}(t_1,t_2,t_3) = \EE e^{i(t_1(bf-ce) + t_2(ae-bd) + t_3 (af-cd))}.
\ee

We notice that the expression:
$$ t_1(bf-ce) + t_2(ae-bd) + t_3 (af-cd) = (a,b,c,d,e,f)^T Q (a,b,c,d,e,f)$$ 
is a symmetric quadratic form in the Gaussian vector:
$$ (a,b,c,d,e,f),$$
where the matrix for the quadratic form is given by:
$$ Q = \frac{1}{2} \begin{bmatrix}
0 & 0 & 0 & 0 & t_1 & t_2 \\
0 & 0 & 0 & -t_1 & 0 & t_3 \\
0 & 0 & 0 & -t_2 & -t_3 & 0 \\
0 & -t_1 & -t_2 & 0 & 0 & 0 \\
t_1 & 0 & -t_3 & 0 & 0 & 0  \\
t_2 & t_3 & 0 & 0 & 0 & 0 \\
\end{bmatrix} .$$

Using \cite[Thm. 2.1]{Scar} while taking $t=1$ (and treating $t_1,t_2,t_3$ as parameters), we have:
\begin{align}
 \EE e^{i(t_1(bf-ce) + t_2(ae-bd) + t_3 (af-cd))} &= \frac{1}{\sqrt{\det(\mathbbm{1} - 2iQ)}} \\
 &= \frac{1}{1 + t_1^2 + t_2^2 + t_3^2}.
\end{align}
We can use this to compute:

\begin{align*}
\rho(x,y,z) &= \frac{1}{(2\pi)^3} \int_\R \int_\R \int_\R \frac{e^{-i(t_1 x  + t_2 y + t_3 z)}}{1+t_1^2 + t_2^2 + t_3^2}\ dt_1 dt_2 dt_3 \\
 &= \frac{1}{(2\pi)^3} \int_0^{2\pi} \int_0^{\pi} \int_0^{\infty} \frac{e^{-i \left|(x,y,z)\right| r \cos \phi}}{1+r^2} r^2 \sin \phi\ dr d\phi d\theta \\
&= \frac{1}{(2\pi)^2} \int_0^{\pi} \int_0^{\infty} \frac{e^{-i \left|(x,y,z)\right| r \cos \phi}}{1+r^2} r^2 \sin \phi \ dr d\phi \\
&= \frac{1}{(2\pi)^2} \int_0^{\infty} \frac{r^2}{1+r^2}  \frac{e^{-i \left|(x,y,z)\right| r \cos \phi}}{i |(x,y,z)| r} \bigg|_{\phi=0}^{\phi=\pi} dr \\
&= \frac{1}{2\pi^2} \frac{1}{|(x,y,z)|} \int_0^{\infty} \frac{r}{1+r^2}  \sin( |(x,y,z)| r) \ dr \\
&= \frac{1}{4\pi} \frac{e^{-|(x,y,z)|}}{|(x,y,z)|} .\\
\end{align*}

Thus, the expectation becomes:
\begin{align*}
\EE |\det \hat{J}_3 | &= \frac{1}{4\pi} \int_{\R^3} |2xz - y^2| \frac{e^{-|(x,y,z)|}}{|(x,y,z)|}\ dx dy dz \\
&= \frac{1}{8\pi} \int_{\R^3} |a_1 a_3 - a_2^2| \frac{e^{-\sqrt{\frac{a_1^2}{2} + a_2 + \frac{a_3^2}{2} }}}{\sqrt{\frac{a_1^2}{2} + a_2 + \frac{a_3^2}{2} }} \ da_1 da_2 da_3 
\end{align*}
where we have made the change of variables
$a_1 = \sqrt{2} x, a_2 = y, a_3 = \sqrt{2} z$.
Let us view $a_1,a_2,a_3$ as the entries of a symmetric matrix:
$$ A = \begin{bmatrix}
a_1 & a_2 \\
a_2 & a_3 
\end{bmatrix}.$$
Consider the coordinates given by the eigenvalues $\lambda_1, \lambda_2$ of $A$
along with the angle $\alpha \in [0,\pi/2)$
associated with the orthogonal transformation that diagonalizes $A$:
$$ M = \begin{bmatrix}
\cos \alpha & -\sin \alpha \\
\sin \alpha & \cos \alpha
\end{bmatrix}.$$
The Jacobian determinant for changing coordinates to $(\lambda_1,\lambda_2,\alpha)$ is $|\lambda_1 - \lambda_2|$.
We recognize $\frac{a_1^2}{2} + a_2 + \frac{a_3^2}{2} = \frac{\lambda_1^2 + \lambda_2^2}{2}$
as half the Frobenius norm of $A$, and $a_1 a_3 - a_2^2 = \lambda_1 \lambda_2$
as the determinant of $A$.
With respect to these coordinates, we have:
\begin{align*}
\EE |\det \hat{J}_3 |  &= \frac{\sqrt{2}}{8\pi} \int_0^{\pi/2}\int_{\R^2} |\lambda_1 \lambda_2| \frac{e^{-\sqrt{\frac{\lambda_1^2 + \lambda_2^2}{2}}}}{\sqrt{\lambda_1^2 + \lambda_2^2 }} |\lambda_1 - \lambda_2|\ d\lambda_1 d \lambda_2 \ d \alpha \\
&= \frac{\sqrt{2}}{16} \int_{\R^2} |\lambda_1 \lambda_2| \frac{e^{-\sqrt{\frac{\lambda_1^2 + \lambda_2^2}{2}}}}{\sqrt{\lambda_1^2 + \lambda_2^2 }} |\lambda_1 - \lambda_2|\ d\lambda_1 d \lambda_2 \\
&= \frac{\sqrt{2}}{16} \int_0^{2 \pi}\int_0^\infty r^3 |\cos \theta \sin \theta ( \cos \theta - \sin \theta)| e^{-r/\sqrt{2}}\ dr\ d \theta,\\
&= \frac{3\sqrt{2}}{2} \int_0^{2 \pi} |\cos \theta \sin \theta ( \cos \theta - \sin \theta)|\ d \theta,\\
&= \frac{3\sqrt{2}}{2} \left( \frac{8 - 2 \sqrt{2}}{3} \right) = 4 \sqrt{2} - 2
\end{align*}
where we have utilized polar coordinates $\lambda_1 = r \cos \theta, \lambda_2 = r \sin \theta$. Substituting the obtained number into \eqref{eq:E3} gives the desired result $E_3=6 \sqrt{2} - 3$.


\end{proof}

\section{The complex case}
\subsection{Complex Kostlan polynomials}
We consider now the space
$\C[x_0,x_1,\ldots, x_n]_{(d)}$ 
 of homogeneous polynomials of degree $d$ in $n+1$ variables.
 A complex Kostlan polynomial is obtained by replacing real Gaussian variables with complex Gaussian variables in the definition from section \ref{sec:kostlan}. Specifically we take:
$$f(x)=\sum_{|\alpha|=d}\xi_\alpha x^{\alpha}, \quad \xi_\alpha\sim N_\C\left(0,\binom{d}{\alpha} \right),$$
where as before the $\xi_\alpha$ are independent. 
The resulting probability distribution on the space $\C[x_0,x_1,\dots, x_n]_{(d)}$ is invariant by the action of $U(n+1)$ by change of variables (see \cite{Kostlan95, BCSS}).
\subsection{A general construction}
We will need the following elementary Lemma.
\begin{lemma}\label{lemma:A}Consider the $2m\times 2m$ real matrix $\tilde{A}$ defined by:
\be\tilde{A}=\left[\begin{array}{ccccc}A_{11} &   & \cdots &  & A_{1m} \\ &  &  &  &  \\\vdots &  &  &  & \vdots \\ &  &  &  &  \\A_{m1} &  & \cdots &  & A_{mm}\end{array}\right]\quad \textrm{where} \quad A_{ij}=\left[\begin{array}{cc}a_{ij} & b_{ij} \\-b_{ij} & a_{ij}\end{array}\right] \ee
and the $m\times m$ complex matrix $A^{\C}$ defined by:
\be A^\C=\left[\begin{array}{ccccc}a_{11}+ib_{11} &   & \cdots &  & a_{1m}+ib_{1m} \\ &  &  &  &  \\\vdots &  &  &  & \vdots \\ &  &  &  &  \\a_{m1}+ib_{m1} &  & \cdots &  & a_{mm}+ib_{mm}\end{array}\right].\ee
Then:
\be \det(\tilde{A})=\det(A^\C)\ \overline{\det(A^\C)} \ \ {\it \ and \  \ }\det(\tilde{A})\geq 0.\ee
\end{lemma}
\begin{proof}Let $P\in GL(2m, \R)$ be the permutation matrix corresponding to the permutation:

\be \sigma=\left(\begin{array}{cccccccccc}1 & 2 & \cdots &m-1& m & m+1 & m+2 & \cdots & 2m-1 & 2m \\1 & 3 & \cdots &2m-3& 2m-1 & 2 & 4 & \cdots & 2m-2 & 2m\end{array}\right).\ee
Then the matrix $P^{-1}\tilde{A}P$ is of the form:
\be \left[\begin{array}{cc}M & N \\ -N & M\end{array}\right]\quad \textrm{where}\quad N=(a_{jl})\quad \textrm{and}\quad M=(b_{jl}).\ee
Note that $A^\C=M+iN$ and that:
\be\left[\begin{array}{cc}\mathbbm{1} & 0 \\ -i\mathbbm{1} & 1\end{array}\right]\cdot \left[\begin{array}{cc}M & N \\ -N & M\end{array}\right]\cdot\left[\begin{array}{cc}\mathbbm{1} & 0 \\ i\mathbbm{1} & 1\end{array}\right] =\left[\begin{array}{cc}M+iN & N \\ 0 & M-iN\end{array}\right] ,\ee
which implies $\det(P^{-1}\tilde{A}P)=\det(M+iN)\det(M-iN)$. Consequently:
\begin{align}\det(\tilde{A})&=\det(P^{-1}\tilde{A}P)\\
&=\det(M+iN)\det(M-iN)\\
&=\det(A^\C)\ \overline{\det(A^\C)}.
\end{align}

\end{proof}

Consider now a random vector $w=(w_1, \ldots, w_{2n-3})$ in $\C^{2n-3}$ whose entries are independent Gaussian variables distributed as: 
\be w_j\sim N_{\C}\left(0,\binom{2n-4}{ j-1}\right)\quad j=1, \ldots,  2n-3.\ee
In other words:
\be w_j\sim\sqrt{\frac{1}{2}\binom{2n-4}{j-1}}(\xi_1+i\xi_2)\ee where $\xi_1, \xi_2$ are two standard independent Gaussians.

Let  $w^{(1)}, \ldots, w^{(n-1)}$ be independent random vectors all distributed as $w$. We define the random $(2n-2)\times (2n-2)$ matrix $\hat{J}_n^\C$ as:
\be \hat{J}_n^\C=\left[\begin{array}{ccccc}w_1^{(1)} & 0 & \ldots & w_{1}^{(n-1)} & 0 \\w_2^{(1)} & w_1^{(1)} &  & w_2^{(n-1)} & w_{1}^{(n-1)} \\\vdots & w_2^{(1)} &  & \vdots & w_2^{(n-1)} \\w_{2n-3}^{(1)} & \vdots &  & w_{2n-3}^{(n-1)} & \vdots \\0 & w_{2n-3}^{(1)} & \ldots & 0 & w_{2n-3}^{(n-1)}\end{array}\right]\ee

\begin{thm}\label{thm:averagecomplex}The number $C_n$ of lines on a generic hypersurface of degree $2n-3$ in $\C\emph{\textrm{P}}^n$ is:
\be C_n=\left( \frac{(2n-3)^{2n-2}}{\Gamma(n)\Gamma(n+1)}\prod_{k=0}^{2n-3}\binom{2n-3}{k}^{-1}\right)\EE|\det \hat{J}^\C|^2\ee
\end{thm}
\begin{proof}(sketch) The proof proceeds similarly to the proof of Theorem \ref{thm:averagereal}; again we use a general approach and then specialize to the case of the Grassmannian of lines in $\CP^n$.

We view the complex Grassmannian $Gr_\C(k,m)$ as a real manifold of dimension $2k(m-k)$ and the space $\C[z_1, \ldots, z_m]_{(d)}$ of homogeneous polynomials of degree $d$ as a real vector space of dimension $2\binom{m+d-1}{d}$:
\be f(z_1, \ldots, z_m)=\sum_{|\alpha|=d}(a_\alpha+i b_\alpha)z_1^{\alpha_1}\cdots z_m^{\alpha_m}, \quad a_\alpha, b_\alpha \in \R.\ee
We put a Riemannian structure on $Gr_\C(k,m)$ by declaring the quotient map $U(m)\to Gr_\C(k,m)$ to be a Riemannian submersion.

Consider the bundle $\textrm{sym}^d(\tau^*_{k,m})$ on $Gr_\C(k,m)$; a polynomial $f\in \C[z_1, \ldots, z_m]_{(d)}$ defines a section $\sigma_f$ of this bundle and, in the case $k(m-k)=\binom{k+d-1}{d}$, the number of zeros of $\sigma_f$ coincides with the number of $k$-planes on $\{f=0\}$. We build a random section of $\textrm{sym}^d(\tau^*_{k,m})$ by taking $f$ to be a \emph{complex} Kostlan polynomial:
\be f(z)=\sum_{|\alpha|=d}\xi_\alpha z_1^{\alpha_1}\cdots z_m^{\alpha_m}\ee
where the $\xi_\alpha$ are independent and distributed as:
\be\label{eq:xia} \xi_\alpha=\sqrt{\frac{d!}{\alpha_1!\cdots \alpha_m!}}\sqrt{\frac{1}{2}}(\xi_1+i\xi_2)\ee
and $\xi_1, \xi_2$ are standard, independent Gaussians. In this way the resulting probability distribution on the space of polynomials is invariant by the action of the unitary group $U(m)$.

Note that in the case $k=2, m=n+1$, with probability one the number of zeros of $\sigma_f$ equals the number of $k$-planes on a generic hypersurface of degree $d$ in $\CP^n$.

To trivialize the bundle we proceed as in the proof of Theorem \ref{thm:averagereal}, using now complex variables.
The invariance of $f$ under the action of the unitary group allows to reduce the computation for the Kac-Rice density at a point (which we again assume is $w_0=\textrm{span}\{e_1, \ldots, e_k\}$). We obtain:
\be \EE\#\{w\in Gr_\C(k,m)\,|\, \sigma_f(w)=0\}=|Gr_\C(k,m)|\rho_\C\ee
where now:
\be \rho_\C= \EE\left\{|\det \tilde{J}(w_0)|\,\big|\, \sigma_f(w_0)=0\right\}p^{\C}(0;w_0),\ee
with $p^{\C}(0;w_0)$ the density at zero of the vector of the real coefficients of $\tilde{\sigma}_f|_{w_0}\in \C[z_1, \ldots, z_k]_{(d)}$, and the matrix
\be \tilde J(w_0)=\left(\nabla_i(\tilde{\sigma}_f)_j(w_0)\right)\ee
the $2k(m-k)\times 2k(m-k)$ matrix of the derivatives of the coordinates of  $\sigma_f$ with respect to an orthonormal frame field at $w_0$.

Note that $\tilde{\sigma}_{f}({w_0})=f|_{w_0}$ is the random polynomial:
\be \sigma_f({w_0})(z_1, \ldots, z_k)=\sum_{|\alpha|=d}\xi_\alpha z_1^{\alpha_1}\cdots z_k^{\alpha_k}\ee
from which we immediately see that in the case $k=2, m=n+1$, for the Grassmannian $Gr_\C(2, n+1)$ of lines in $\CP^n$ we have:
\be p^\C(0, w_0)=\prod_{k=0}^{2n-3}\frac{1}{\pi}\binom{2n-3}{k}^{-1}.\ee
For the computation of $\tilde J(w_0)$ we use the orthonormal basis of $T_{w_0}Gr_\C(k, m)$ given by derivatives at zero of the curves:
\be \gamma_{lj}^1, \gamma_{lj}^2:(-\epsilon, \epsilon)\to Gr_\C(2,m)\ee defined for $l=1, \ldots, k$ and $j=k+1, \ldots, m$ by:
 \be \gamma_{lj}^1(s)=e^{sE_{lj}}e_1\wedge \cdots\wedge e^{sE_{lj}}e_k\quad\textrm{and}\quad \gamma_{lj}^2(s)=e^{i sE_{lj}}e_1\wedge \cdots\wedge e^{i sE_{lj}}e_k.\ee
 It is immediate to verify that:
 \begin{align} \frac{d}{ds}\left(\tilde{\sigma}_{f}(\gamma^1_{kj}(s))(z)\right)&=\sum_{\alpha_1+\cdots+\alpha_k+1=d} \xi_\alpha z_1^{\alpha_1}\cdots z_i^{\alpha_i+1}\cdots z_k^{\alpha_k}\\
 &=\sum_{|\beta|=d}q^{(1)}_\beta z_1^{\beta_1}\cdots z_k^{\beta_k}
 \end{align}
 and that:
 \begin{align}   \frac{d}{ds}\left(\tilde{\sigma}_{f}(\gamma^2_{kj}(s))(z)\right)&=\sum_{\alpha_1+\cdots+\alpha_k+1=d} i \xi_\alpha z_1^{\alpha_1}\cdots z_i^{\alpha_i+1}\cdots z_k^{\alpha_k}\\
 &=\sum_{|\beta|=d}q^{(2)}_\beta z_1^{\beta_1}\cdots z_k^{\beta_k}
  \end{align}
 From this we see that the coefficients $q_\beta^{(1)}$ and $q_\beta^{(2)}$ are distributed as:
 \be\label{eq:xic1} q^{(1)}_\beta=\xi_{\beta_1, \ldots, \beta_{i-1}, \beta_i+1, \beta_{i+1}, \ldots, \beta_k, 0, \ldots, 0, 1, \ldots, 0}\quad \textrm{(there is a $1$ in position $j$)}\ee
 and:
 \be \label{eq:xic2} q_\beta^{(2)}=i \cdot \xi_{\beta_1, \ldots, \beta_{i-1}, \beta_i+1, \beta_{i+1}, \ldots, \beta_k, 0, \ldots, 0, 1, \ldots, 0}\quad \textrm{(there is a $1$ in position $j$)}.\ee 
 In particular, we deduce from \eqref{eq:xic1} and \eqref{eq:xic2} that $\det \tilde{J}(w_0)$ and $\sigma_f(w_0)$ are independent, and consequently:
 \be \EE\left\{|\det \tilde{J}(w_0)|\,\big|\, \sigma_f(w_0)=0\right\}= \EE|\det \tilde{J}(w_0)|.\ee
Recalling the definition \eqref{eq:xia}, and specializing to the case $k=2, m=n+1$ of the Grassmannian of lines, we see again from  \eqref{eq:xic1} and \eqref{eq:xic2} that  the matrix $\tilde J(w_0)$ has the same shape as the matrix $\tilde A$ from Lemma \ref{lemma:A}. Collecting a factor of $\sqrt{d}=\sqrt{2n-3}$ from each row and using Lemma \ref{lemma:A} we obtain:
\be  \EE|\det \tilde{J}(w_0)|=(2n-3)^{2n-2}\EE\left(\det \hat{J}^\C_n\ \overline{\det \hat{J}_n^\C}\right).\ee
Putting all the pieces together, and using the formula $|Gr_{\C}(2, n+1)|=\frac{\pi^{2n-2}}{\Gamma(n)\Gamma(n+1)}$ (see Remark \ref{remark:vg}), we get:
\begin{align} C_n&=|Gr_\C(2, n+1)|\ p^\C(0, w_0) \ \EE|\det \tilde{J}(w_0)|\\
&=\frac{\pi^{2n-2}}{\Gamma(n)\Gamma(n+1)}\prod_{k=0}^{2n-3}\frac{1}{\pi}\binom{2n-3}{k}^{-1}(2n-3)^{2n-2}\EE\left(\det \hat{J}^\C_n\ \overline{\det \hat{J}_n^\C}\right)\\
&=\left( \frac{(2n-3)^{2n-2}}{\Gamma(n)\Gamma(n+1)}\prod_{k=0}^{2n-3}\binom{2n-3}{k}^{-1}\right)\EE|\det \hat{J}^\C_n|^2.
\end{align}
  \end{proof}
\begin{remark}[Real versus complex Gaussians]\label{remark1}Consider the matrix:
\be A_n(x)=  \left[\begin{array}{ccccc}x_{1,1} & 0 & \cdots & x_{n-1, 1} & 0 \\\vdots & x_{1,1} &  & \vdots & x_{n-1, 1} \\\binom{2n-4}{j-1}^{1/2}x_{1,j} & \vdots &  & \binom{2n-4}{j-1}^{1/2}x_{n-1,j} & \vdots \\\vdots & \binom{2n-4}{j-1}^{1/2}x_{1,j} &  & \vdots & \binom{2n-4}{ j-1}^{1/2}x_{n-1,j} \\x_{1, 2n-3} & \vdots &  & x_{n-1, 2n-3} & \vdots \\0 & x_{1, 2n-3} & \cdots & 0 & x_{n-1, 2n-3}\end{array}\right].\ee
The determinant $P_{n}(x)$ of $A_n(x)$ is a homogeneous polynomial of degree $D=2n-2$ in $N=(n-1)(2n-3)$ many variables and by construction we have:
\be\label{eq:edet1}  \EE|\det \hat J_n|=\frac{1}{(2\pi)^{N/2}}\int_{\R^N}|P_{n}(x)|\ e^{-\frac{1}{2}\|x\|^2}dx \ee
and:
\begin{align}\label{eq:edet2}\EE|\det \hat{J}^\C_n|^2&=\frac{1}{\pi^N}\int_{\C^N}P_n(z)\overline{P_n(z)}\ e^{-\|z\|^2}dz.
\end{align}
Recall that, given a homogeneous polynomial $ P(z)=\sum_{|\alpha|=D}P_\alpha z_1^{\alpha_1}\cdots z_N^{\alpha_N}$ of degree $D$ in $N$ variables, we have denoted by $\|P\|_B$ its Bombieri norm:
\be\label{eq:bombieri} \|P\|_B=\left(\sum_{|\alpha|=D}|P_\alpha|^2\frac{\alpha_1!\cdots \alpha_N!}{D!}\right)^{\frac{1}{2}}.\ee
Then, it is possible to rewrite \eqref{eq:edet2} as:
\be\label{eq:Bombieri} \EE|\det \hat{J}^\C_n|^2=(2n-2)!\ \|P_{n}\|_B^2.\ee
\end{remark}
\subsection{The $27$ lines on a complex cubic}
\begin{cor}\label{cor:C}
There are $27$ lines on a generic cubic in $\mathbb{C}\emph{P}^3$.
\end{cor}
\begin{proof}This is the case $n=3$ in the previous theorem. We have:
\begin{align} C_3&=|Gr_\C(2,4)|\cdot p^\C(0; w_0)\cdot3^4\cdot \EE|\det \hat{J}^\C_3|^2\\
&=\frac{\pi^4}{12}\cdot\frac{1}{9\pi^4}\cdot 81\cdot \EE|\det \hat{J}^\C_3|^2\\
&=\frac{3}{4}\EE|\det \hat{J}^\C_3|^2\,
\end{align}
 For the computation of $\EE|\det \hat{J}^\C_3|^2$ we use \eqref{eq:Bombieri}. We have the following expression for $P_3(x)=\det A_3(x)$:
\be x_{13}^2x_{21}^2-2x_{12}x_{13}x_{21}x_{22}+2x_{11}x_{13}x_{22}^2-2x_{12}^2x_{21}x_{23}-2x_{11}x_{13}x_{21}x_{23}-2x_{11}x_{12}x_{22}x_{23}+x_{11}^2x_{23}^2.\ee
Recalling \eqref{eq:bombieri} we can immediately compute the Bombieri norm of $P_3$:
\be \|P_3\|_B^{2}=\frac{36}{4!}.\ee
From this we get $\EE|\det \hat{J}^\C_3|^2=4!\|P_3\|^2_B=36,$ and consequently $C_3=27$.
\end{proof}

\section{Asymptotics}
The main purpose of this section is to prove Theorem \ref{thm:sqrt}, which gives the asymptotic \eqref{eq:sqrt} of $E_n$ in the logarithmic scale, as discussed above in Section \ref{sec:discussion} (the ``square root law'').
This will follow from a combination of the lower bound given in  Corollary \ref{cor:signed} and the upper bound that we will prove in Proposition \ref{propo:upper}.

\subsection{The upper bound}
The strategy of our proof can be described as follows.  In order to simplify notations, 
let us absorb the variances in the variables of the matrix $A_n(x)$ 
and consider it as the matrix $B_n(u)$ with entries the random variables 
$u_{i,j}=\binom{2n-4}{j-1}x_{i,j}$, so that $\det A_n(x)=\det B_n(u).$ 
\be\label{eq:B} 
B_n(u)=  \left[\begin{array}{ccccc}u_{1,1} & 0 & \cdots & u_{n-1, 1} & 0 \\\vdots & u_{1,1} &  & \vdots & u_{n-1, 1} \\u_{1,j} & \vdots &  & u_{n-1,j} & \vdots \\\vdots & u_{1,j} &  & \vdots & u_{n-1,j} \\u_{1, 2n-3} & \vdots &  & u_{n-1, 2n-3} & \vdots \\0 & u_{1, 2n-3} & \cdots & 0 & u_{n-1, 2n-3}\end{array}\right].
\ee
We will use both the double-index notation $u_{i,j}$
as well as single-index notation $u_k$,
$k=1,2,..,N$ for the $N=(n-1) (2n-3)$ 
many variables appearing in $B_n(u)$.

Given a permutation $\sigma \in S_{2n-2}$ we consider the product:
\be \label{eq:mono}\prod_{i=1}^{2n-2}B_n(u)_{\sigma(i),i}=u_1^{\alpha_1}\cdots u_N^{\alpha_N}.\ee 
We will call $u_1^{\alpha_1}\cdots u_N^{\alpha_N}$ the monomial generated by the permutation $\sigma$.
We will denote by $I_1$ the set of all the multi-indices of all possible monomials generated by permutations in $S_{2n-2}$.
In this way we can write:
\be Q_n(u)=\det B_n(u)=\sum_{\alpha\in I_1} Q_\alpha u_1^{\alpha_1}\cdots u_{N}^{\alpha_N}.
\ee
We first prove that for each each permutation $\pi \in S_{2n-2}$, the monomial $\alpha_\pi$ generated by $\pi$ occurs with a non-zero integral coefficient in the polynomial $Q_n(u)$. In other words there are no cancellations occuring in the 
Laplace expansion of the determinant of $B_n(u)$: 
such cancellations are a priori possible since the same monomial can be generated by several permutations,
which is evident \eqref{eq:B} from the structure of the matrix. 
The proof of the above statement follows immediately from Lemma \ref{lemma:I1} proved below. 
We then prove (Lemma \ref{lemma:I2}) that in the expansion of the polynomial $Q_n(u)^2$,
each ``cross-term'' $Q_{\alpha}Q_{\beta}u^\alpha u^\beta$ that appears with a non-zero coefficient can be
``charged'' to some square term $Q_\gamma^2 u^{2\gamma}$. Of course, many cross-terms might be charged to the same square-term, but the number of different pairs $(\alpha,\beta)$ such that $\alpha+\beta = 2 \gamma$, is bounded by some number at most exponential in $n$. 
Together, these two lemmas reduce the
problem of bounding $\EE  (Q_n(u))^2$ (up to a loss of an exponential factor) to the problem of bounding
\be\EE \sum_\gamma Q_\gamma^2 u_\gamma^2,\ee and the latter can be bounded using linearity of expectation in terms of the Bombieri norm of the polynomial $P_n(x) = \det A_n(x)$ (Proposition \ref{propo:upper}).

We now prove the necessary preliminary results required to carry through the argument sketched above.

\begin{lemma}\label{lemma:I1}
For each $\alpha \in I_1$, $Q_\alpha \neq 0$ and $|Q_\alpha|>1.$
\end{lemma}

\begin{proof}[Proof of Lemma]

For a given multi-index $\alpha$ suppose that 
$\sigma = \sigma(1) \sigma(2) \cdots \sigma(2n-2)$ and $\tau = \tau(1) \tau(2) \cdots \tau(2n-2)$ 
are two permutations (each given in one-line notation)
that are associated with the monomial $u^\alpha$ 
in the Laplace expansion of the determinant of $B_n(u)$.
In other words, 
the corresponding terms in the Laplace expansion
are $\sgn(\sigma) u^\alpha $ and $\sgn(\tau)  u^\alpha$, respectively.
It suffices to show, for arbitrary such $\sigma$
and $\tau$, that we have $\sgn(\sigma) = \sgn(\tau)$, or equivalently that $\theta=\tau^{-1}\sigma$ is even.

Note that in the matrix 
$B_n(u)$ each variable appears exactly twice
and in positions that are separated by exactly one
increment in the row and the column values.
The assumption that $\sigma$ and $\tau$ generate the same monomial $u^\alpha$
then implies the following claim.

\begin{claim}\label{claim1}
Suppose $\sigma(j) \neq \tau(j)$.
Then if $j$ is odd, 
$\sigma(j+1) = \tau(j) + 1$ and $\tau(j+1) = \sigma(j) +1$,
and if $j$ is even,
$\sigma(j-1) = \tau(j) - 1$ and $\tau(j-1) = \sigma(j)-1$.
\end{claim}

We will also need the next claim.

\begin{claim}\label{claim2}
Suppose $\tau(j) = \sigma(k)$ for $j \neq k$.
Then the parities of $k$ and $j$ are the same; moreover if $j$ is odd 
$\sigma(j+1) = \tau(j) +1$ and $\tau(k+1) = \sigma(k) + 1$, if  $j$ is even 
$\sigma(j-1) = \tau(j) - 1$ and
$\tau(k-1) = \sigma(k) - 1$.
\end{claim}

We prove Claim \ref{claim2} in the case that $j$ is odd;
the case $j$ is even is similar and is omitted.
The variable in the matrix $B_n(u)$
in position $(\tau(j), j)$
must be selected by $\sigma$ as well.
Since $\sigma(j) \neq \tau(j)$ and $j$ is odd,
the only option is that $\sigma(j+1) = \tau(j) +1.$
We will see that the case $k$ is even leads to a contradiction.
Since $\tau(k) \neq \sigma(k)$ and $k$ is even, 
in order for the variable in position 
$(\sigma(k),k)$ to be selected by $\tau$ 
we must have $\tau(k-1) = \sigma(k) - 1$ (by Claim \ref{claim1}).
This implies that $\sigma(k-1) \neq \tau(k-1)$,
so that there is some $\ell \neq k-1$ such that $\sigma(\ell) = \tau(k-1)$.
In order for the variable in position 
$(\sigma(\ell), \ell)$ to be selected by $\tau$ 
we must have $\tau(\ell-1) = \sigma(\ell) - 1 = \tau(k-1) - 1$
(note that the alternative option $\tau(\ell+1) = \sigma(\ell) + 1$ is prevented since 
$\sigma(\ell)+1 = \tau(k-1) + 1 = \sigma(k) = \tau(j)$).
Iterating this argument $\tau(k-1) - 2$ more steps, 
and recalling equation \eqref{eq:B},
we reach the first row of the matrix $B_n(u)$
where we are forced to select an unpaired variable
contradicting the assumption that $\sigma$ and $\tau$
generate the same monomial.
This shows that $k$ must be odd as well and,
by the same reasoning as above, $\tau(k+1)=\sigma(k)+1$ as stated in the claim.

Let us now go back to the permutation $\theta=\tau^{-1}\sigma.$ 
Claim \ref{claim2} implies that $\theta$ preserves parity, so it can be written as the product of two permutations $\theta=\theta_\textrm{even}\cdot \theta_{\textrm{odd}}$, where $\theta_\textrm{even}$ 
(respectively, $\theta_\textrm{odd}$)
is in the symmetric group $S_{n-1, \textrm{even}}$ (respectively,  $S_{n-1, \textrm{odd}}$) 
on the set of \emph{even}  (respectively, \emph{odd}) numbers 
belonging to $\{1,\ldots,2n-2\}$. We identify  $S_{n-1, \textrm{even}}$ (respectively,  $S_{n-1, \textrm{odd}}$) with the subgroup of $S_{2n-2}$ of permutations which fixes each odd (respectively, even) number in $\{1,\ldots,2n-2\}$.
Claim \ref{claim2} can now be rewritten as:
\be\label{eq:isomo} \theta_{\textrm{odd}}(2k-1)=2j-1\iff \theta_{\textrm{even}}(2j)=2k.\ee
Note that the bijection:
\be\{2,4,\ldots, 2n-2\}\to \{1,3,\ldots, 2n-3\}, \quad 2k\mapsto 2k-1,\ee
induces an isomorphism $\psi:S_{n-1, \textrm{odd}}\to S_{n-1, \textrm{even}}$. 
Equation \eqref{eq:isomo} shows that $\psi(\theta_{\textrm{odd}})=\theta_{\textrm{even}}^{-1}.$ In particular, since the sign of $\theta_{\textrm{even}}$ and $\theta_{\textrm{even}}^{-1}$ are the same it follows that $\theta$ is even.
 \end{proof}
We will also need the following lemma.
\begin{lemma}\label{lemma:I2}Let $I_2$ be the set of all multi-indices $\gamma=(\gamma_1, \ldots, \gamma_N)$ such that there exist $\alpha, \beta\in I_1$ with $\alpha_i+\beta_i=2\gamma_i$ for all $i=1, \ldots, N.$ Then $I_2\subseteq I_1.$
\end{lemma}
\begin{proof}
Let $\sigma$ and $\tau$ be two 
permutations that are associated with the monomials $u^\alpha$ and $u^\beta$, respectively,
in the Laplace expansion of $\det B_n(u)$.
It suffices to construct a third permutation $\omega$
that is associated to the monomial $u^\gamma$.
In order for $u^{\alpha + \beta}$ to be the square of a monomial,
we must have that $\alpha_i + \beta_i$ is either $0,2$, or $4$.
If $\alpha_i + \beta_i = 0$ then $\alpha_i=\beta_i = 0$.
If $\alpha_i + \beta_i = 4$ then $\alpha_i = 2$ and $\beta_i = 2$.
If $\alpha_i + \beta_i = 2$ then there are three possibilities:
we can have $\alpha_i=2$ and $\beta_i=0$ or $\alpha_i=1$ and $\beta_i=1$
or $\alpha_i=0$ and $\beta_i=2$.
In terms of $\sigma$ and $\tau$,
for each pair of columns of $B_n(u)$
with column numbers $2k-1$ and $2k$,
only the following three cases can occur.
\begin{enumerate}[{Case} 1.]
 \item
\label{itemlabel:case1}

We have $\sigma(2k-1) = \tau(2k-1)$,
which implies $\sigma(2k) = \tau(2k)$.

\item
\label{itemlabel:case2}
We have $\sigma(2k-1) \neq \tau(2k-1)$
(which implies $\sigma(2k) \neq \tau(2k)$) and $\tau(2k)=\sigma(2k-1)+1$ (which implies $\sigma(2k)=\tau(2k-1)+1$).

\item
\label{itemlabel:case3}
We have $\sigma(2k-1) \neq \tau(2k-1)$
(which implies $\sigma(2k+2) \neq \tau(2k+2)$),
and $\tau(2k)=\tau(2k-1)+1$ and $\sigma(2k)=\sigma(2k-1)+1.$
\end{enumerate}

Now we construct $\omega$ while, for each pair of column numbers, 
basing our choice for $\omega(2k-1)$ and $\omega(2k)$ 
in terms of the three cases.
In each of Cases \ref{itemlabel:case1}  and \ref{itemlabel:case2}, we simply take $\omega(2k-1) = \sigma(2k-1)$
and $\omega(2k) = \sigma(2k)$.
In the remaining Case \ref{itemlabel:case3},
we take $\omega(2k-1) = \sigma(2k-1)$ and $\omega(2k) = \tau(2k)$.

In order to see that $\omega$ is indeed a permutation,
it suffices to see that it is one-to-one.
Since $\sigma$ is a permutation, 
the only case that requires checking is
when we select $\omega(2k) = \tau(2k)$ in Case 3. We will need the following claim.

\begin{claim}\label{claim3}Assume a pair of columns with column numbers $2k-1$ and $2k$ falls into Case \ref{itemlabel:case2}  (respectively Case \ref{itemlabel:case3} ). Then there exists another pair of columns that falls into Case \ref{itemlabel:case2}  (respectively Case \ref{itemlabel:case3} ) with column numbers $2j-1$ and $2j$ with $j\neq k$ such that: $\sigma(2k-1)=\tau(2j-1)$ and $\sigma(2j)=\tau(2k)$.
\end{claim}

We note that this claim can be verified pictorially
from the equation \eqref{eq:B} of the matrix $B_n(u)$ 
by ``chasing'' the row-column structure.
However, we will proceed formally.
We prove the claim simultaneously 
for Case \ref{itemlabel:case2}  and Case \ref{itemlabel:case3}  
by induction on $q=\sigma(2k-1).$
For the base of the induction, let $\sigma(2k-1)=1$. Consider $j$ such that $\tau(j)=1$. 
Observe that $j$ must be odd, 
and the variable  in position $(j, 2)$ must also be selected. 
If the pair of columns $p=(c_{2k-1},c_{2k})$ falls into Case \ref{itemlabel:case2}  
(respectively Case \ref{itemlabel:case3} ), this forces the pairs of columns $(c_j, c_{j+1})$ to fall into Case \ref{itemlabel:case2}  (respectively Case \ref{itemlabel:case3} ). 

Assume now that the Claim holds for $\sigma(2k-1)\leq q.$ Pick $j$ such that $\tau(j)=\sigma(2k-1)=q+1$.
If $(c_{2k-1}, c_{2k})$ falls into Case \ref{itemlabel:case2}, then $j$ must be odd. In fact if it was even the pair $p'=(c_j, c_{j-1})$ would give (by inductive hypothesis) another pair $p''$ which is in the same Case as $p'$; this contradicts the injectivity of $\sigma$. Now, if $j$ is odd, then $\sigma(j+1)=\tau(j)+1$, otherwise we would fall into Case \ref{itemlabel:case3}  contradicting again the injectivity of $\tau$.

If now $(c_{2k-1}, c_{2k})$ falls into Case \ref{itemlabel:case3}, a similar argument shows that  $p'=(c_j, c_{j-1})$ must also fall into Case \ref{itemlabel:case3}.
This proves the claim.

Let us now finish verifying the injectivity of $\omega$. Consider the pair of columns $(c_{2k-1}, c_{2k})$, which falls into Case \ref{itemlabel:case3}. Claim \ref{claim3} produces  now a new pair $(c_{\ell}, c_{\ell+1})$ that also falls into Case \ref{itemlabel:case3}, such that $\ell=2j-1$, $\sigma(2k-1)=\tau(2j-1)$ and $\sigma(2j)=\tau(2k)$. By definition $\omega(2j-1)=\sigma(2j-1)$ and $\omega(2j)=\tau(2j)\neq \sigma(2j).$ Hence $\omega$ is injective.

Considering the variables that are selected by $\omega$ it is simple to check that it is associated with $\gamma$ as desired.
\end{proof}
We will now prove the claimed upper bound for $E_n$.
\begin{prop}\label{propo:upper}Using the notation of Theorem \ref{thm:averagereal} and Theorem \ref{thm:averagecomplex}, there exists $b>1$ such that:
\be E_n\leq \sqrt{n}b^n C_n^{1/2}.\ee
\end{prop}

\begin{proof}
As a first step, we claim that there exists $b>1$ such that:
\be \label{eq:mainin} \EE | \det \hat{J}^\R_n  | \leq b^n\left(\EE |\det \hat{J}_n^\C|^2\right)^{1/2}.\ee
We want to apply the Cauchy-Schwarz inequality, and estimate the quantity:
\be \EE | \det \hat{J}^\R_n  | = \EE|\det B_n(u)|\leq \left(\EE \left(\det B_n(u)\right)^2\right)^{1/2}=\left( \EE Q_n(u)^2 \right)^{1/2}.\ee
We write now:
\be \label{eq:square}Q_n(u)^2=\sum_{\gamma\in I_1}Q_\gamma^{2}u_1^{2\gamma_1}\cdots u_N^{2\gamma_N}+\sum_{\alpha\neq \beta} Q_\alpha Q_\beta u_1^{\alpha_1+\beta_1}\cdots u_{N}^{\alpha_N+\beta_N},\ee
where $I_1$ is the index set defined in Lemma \ref{lemma:I1}.
Observe that, after taking expectation and using independence, in the second sum in \eqref{eq:square} only terms such that $\alpha_i+\beta_i$ is even for all $i=1, \ldots, N$ give a nonzero contribution:
\be\label{eqe1} \EE Q_n(u)^2=\sum_{\gamma\in I_1}Q_\gamma^{2}\EE u_1^{2\gamma_1}\cdots u_N^{2\gamma_N}+\sum_{\alpha\neq \beta\,\textrm{and  $\alpha+\beta$ ``even''}} Q_\alpha Q_\beta \EE u_1^{\alpha_1+\beta_1}\cdots u_{N}^{\alpha_N+\beta_N}.\ee
We now rewrite the double sum on the right as:
\be \sum_{\alpha\neq \beta\,\textrm{and  $\alpha+\beta$ even}} Q_\alpha Q_\beta \EE u_1^{\alpha_1+\beta_1}\cdots u_{N}^{\alpha_N+\beta_N}=\sum_{\gamma\in I_1}\left(\sum_{\alpha+\beta =2\gamma}Q_\alpha Q_\beta\right)u_1^{2\gamma_1}\cdots u_N^{2\gamma_n}\ee
where $I_2$ is the index set defined in Lemma \ref{lemma:I2}.

Note that there exists $b_1>1$ such that $|Q_\gamma|\leq b_1^n$ for every $\gamma\in I_1$. In fact there are only $O(1)^n$ many possible ways a given monomial can appear as one of the summands in the Laplace expansion of $\det B_n(u).$ Moreover there exists $b_2>1$ such that, for every fixed $\gamma$, the cardinality of the set of pairs $(\alpha,\beta)$ such that $\alpha+\beta =2\gamma$ is bounded by $b_2^n$. In fact, given $\gamma_i$ there are at most six possible pairs for $(\alpha_i, \beta_i)$ such that $\alpha_i+\beta_i=2\gamma_i$ (namely $(0,0)$, $(0, 2)$, $(1,1)$, $(1,2)$,  $(2,0)$ and $(2,2)$). In our case each monomial $u_1^{2\gamma_1}\cdots u_N^{2\gamma_n}$ can have at most $2n-2$ many variables with nonzero exponents, hence combinatorially we have at most $6^{2n-2}\leq b_2^n$ many pairs $(\alpha, \beta)$ with $\alpha+\beta=2\gamma$, as claimed.
As a consequence we can bound:
\be\label{eqe2} \sum_{\gamma\in I_2}\left(\sum_{\alpha+\beta =2\gamma}Q_\alpha Q_\beta\right)u_1^{2\gamma_1}\cdots u_N^{2\gamma_n}\leq (b_1b_2)^n\sum_{\gamma\in I_2} u_1^{2\gamma_1}\cdots u_N^{2\gamma_n}.\ee
We now use the fact that in the expansion of $\det B_n(u)$ no coefficient $Q_\gamma$ is zero for $\gamma\in I_1$ (by Lemma \ref{lemma:I1}; moreover $|Q_\gamma|\geq 1$ and as a consequence we can write:
\be \sum_{\gamma\in I_2} u_1^{2\gamma_1}\cdots u_N^{2\gamma_n}\leq \sum_{\gamma\in I_1} u_1^{2\gamma_1}\cdots u_N^{2\gamma_n}\leq \sum_{\gamma\in I_1} Q_\gamma^2 u_1^{2\gamma_1}\cdots u_N^{2\gamma_n}.\ee
In the first inequality we have used the fact that $I_2\subseteq I_1$ (Lemma \ref{lemma:I2}).
Combining this with \eqref{eqe1} and \eqref{eqe2} we get:
\be\label{eqe3} \EE Q_n(u)^2\leq b_3^n \sum_{\gamma\in I_1}Q_\gamma^{2}\EE u_1^{2\gamma_1}\cdots u_N^{2\gamma_N}.\ee
We now switch back to the variables $x_1, \ldots, x_N$ (which are standard independent Gaussians).
Recalling the definition of $P_n(x)=\det A_n(x)=\det B_n(u)=Q_n(u)$, we have:
\be \sum_{\gamma\in I_1}Q_\gamma^{2}\EE u_1^{2\gamma_1}\cdots u_N^{2\gamma_N}=\sum_{\gamma \in I_1} P_\gamma^2 \EE x_1^{2\gamma_1}\cdots x_N^{2\gamma_N}.\ee
We now look at $\EE x_1^{2\gamma_1}\cdots x_N^{2\gamma_N}$ for $\gamma\in I_1.$ Using independence:
\be \EE x_1^{2\gamma_1}\cdots x_N^{2\gamma_N}=\prod_{i=1}^N\EE x_i^{2\gamma_i}=\prod_{\{i\,|\,\gamma_i\neq 0\}} \EE x_i^{2\gamma_i}\leq b_4^n\prod_{\{i\,|\,\gamma_i\neq 0\}}\gamma_i!=b_4^n\prod_{i=1}^N\gamma_i! .\ee
For the inequality in the line above we have used the fact that 
only $2n-1$ many variables appear with a nonzero power, 
and for those variables we have $\gamma_i\leq 2$ which implies that 
the moment $\EE x_i^{2\gamma_i} = (2 \gamma_i - 1)!! \leq 2  \gamma_i ! $
(so we can take say $b_4 = 4$).
In particular, continuing from \eqref{eqe3} we get:
\begin{align} \EE P_n(x)^2=\EE Q_n(u)^2&\leq b_3^n \sum_{\gamma\in I_1}P_\gamma^{2}\EE x_1^{2\gamma_1}\cdots x_N^{2\gamma_N}\\
&\leq b_5^n\sum_{\gamma\in I_1}P_\gamma^2 \gamma_1!\cdots \gamma_N!\\
&=b_5^n(2n-2)!\sum_{\gamma}P_\gamma^2\frac{\gamma_1!\cdots \gamma_N!}{(2n-2)!}\\
&=b_5^n(2n-2)!\|P_n\|^2_{B}.
\end{align}
Recalling \eqref{eq:Bombieri}, we have $(2n-2)!\ \|P_{n}\|_B^2=\EE|\det \hat{J}^\C_n|^2,$ which finally implies \eqref{eq:mainin}:
\be \EE | \det \hat{J}^\R_n  | \leq b^n\left(\EE |\det \hat{J}_n^\C|^2\right)^{1/2}.\ee
We can finally use Theorem \ref{thm:averagereal}, and estimate the quantity $E_n$ as:
\begin{align} E_n&=\left(\frac{(2n-3)^{n-1}}{\Gamma(n)}\prod_{k=0}^{2n-3}\binom{2n-3}{k}^{-1/2}\right)\EE|\det \hat J_n^\R|\\
&=\sqrt{n}\left( \frac{(2n-3)^{2n-2}}{\Gamma(n)\Gamma(n+1)}\prod_{k=0}^{2n-3}\binom{2n-3}{k}^{-1}\right)^{1/2}\EE|\det \hat J_n^\R|\\
&\leq \sqrt{n}\left( \frac{(2n-3)^{2n-2}}{\Gamma(n)\Gamma(n+1)}\prod_{k=0}^{2n-3}\binom{2n-3}{k}^{-1}\right)^{1/2} b^n\left(\EE |\det \hat{J}_n^\C|^2\right)^{1/2}\\
&\leq \sqrt{n}b^n C_n^{1/2}.
\end{align}
This proves the statement in the proposition.

\end{proof}

\subsection{The square root law}
\begin{thm}\label{thm:sqrt}Using the notation of Theorem \ref{thm:averagereal} and Theorem \ref{thm:averagecomplex}, we have:
\be \lim_{n\to \infty}\frac{\log E_n}{\log C_n}=\frac{1}{2}.\ee
\end{thm}
\begin{proof}
Applying Corollary \ref{cor:signed} and Proposition \ref{propo:upper} 
we get:
\be (2n-3)!!\leq E_n\leq b^n\sqrt{n} C_n^{1/2}.\ee
We note that $\log (2n-3)!!=n \log(n)+O(n)$ and that, by \eqref{thm:DZ}, $\log C_n=2n \log(n)+O(n).$ As a consequence:
\be n\log(n)+O(n)=\log (2n-3)!!\leq \log E_n\leq \frac{1}{2}\log C_n+n\log b+O(\log n)\ee 
and, dividing by $\log C_n$:
\be \frac{n\log(n)+O(n)}{2n \log(n)+O(n)}\leq \frac{\log E_n}{\log C_n}\leq \frac{1}{2}+\frac{n \log b+O(\log n)}{2n \log (n)+O(n)}.\ee
Taking the limit $n\to \infty$ yields the result.
\end{proof}


\bibliographystyle{plain}
\bibliography{Cubic}
\end{document}